\documentclass[12pt]{amsart}

\usepackage{amsmath}

\def\lr{\mbox{\begin{picture}(7,10)
\put(0,8){\line(1,-1){10}}
\put(0,8){\line(1,0){10}}
\put(10,8){\line(0,-1){10}}
\end{picture}
}}

\usepackage{tikz}
\usepackage{caption}
\usepackage{subcaption}
\usepackage[colorlinks=true, pdfstartview=FitV, linkcolor=blue, citecolor=blue, plainpages=false, pdfpagelabels=true, urlcolor=blue]{hyperref}

\DeclareMathOperator{\sech}{sech}

\usepackage{graphicx}
\usepackage{color}
\usepackage{amsmath}
\usepackage{amssymb}
\newcommand{\beq}{\begin{eqnarray*}}
\newcommand{\feq}{\end{eqnarray*}}
\newcommand{\beqn}{\begin{eqnarray}}
\newcommand{\feqn}{\end{eqnarray}}

\newtheorem{theorem}{Theorem}[section]
\newtheorem{lemma}[theorem]{Lemma}

\theoremstyle{definition}

\newtheorem{example}[theorem]{Example}

\theoremstyle{remark}

\numberwithin{equation}{section}
\allowdisplaybreaks

\begin{document}

%\title[Wave breaking condition for Whitham type equation]{Improved wave breaking condition for the non-local Whitham type equation}
\title[Wave breaking condition for Whitham type equation]{An extension of Seliger's wave breaking condition for the nonlocal Whitham type equation}
%\title[Wave breaking of Whitham type equation]{Non-local Whitham type equation: an extension of Seliger's wave breaking condition}
%\title[Global solutions for Euler-Poisson system]{Global solutions for the two-dimensional Euler-Poisson system with attractive forcing}
%{Finite time blow up conditions for two-dimensional weakly restricted Euler-Poisson equations}
%{A CONSERVATION LAW WITH NONLOCAL FLUX AND ITS APPLICATIONS}
%    Information for first author
\author{Yongki Lee}
%    Address of record for the research reported here
\address{Department of Mathematical Sciences, Georgia Southern University, Statesboro,  30458}
\email{yongkilee@georgiasouthern.edu, WEB: sites.google.com/site/yongkimath/}

%    \thanks will become a 1st page footnote.
%\thanks{The first author was supported in part by NSF Grant \#000000.}
\begin{abstract}

We extend the wave breaking condition in Seliger's work [Proc. R. Soc. Lond. Ser. A., 303 (1968)], which has been used widely to prove wave breaking  phenomena for nonlinear nonlocal shallow water equations.

\end{abstract}
\maketitle

%admits global smooth solutions for a large set of initial
%configurations, so called sub-critical conditions which are not necessarily confined to
%any prefered small neighborhood.

\section{Introduction and statement of main result}
In this paper, we are concerned with the wave breaking phenomena - bounded solutions with unbounded derivatives - for the nonlocal Whitham  type equation \cite{Wh74}:
\begin{equation}\label{whitham_0}
\left\{
  \begin{array}{ll}
   \partial_t u + u \partial_x u + \int_{\mathbb{R}} K (x-\xi)u_{\xi} (t, \xi) \, d\xi =0, & t>0, x \in \mathbb{R}, \\
    u(0,x)=u_0 (x), &   x\in \mathbb{R},\hbox{}
  \end{array}
\right.
\end{equation}
where 
$K(x)=\frac{1}{2\pi}\int_{\mathbb{R}} c(\kappa) e^{i\kappa x} \, d\kappa,$
is the Fourier transform of the desired phase velocity $c(\kappa)$. The function $u(t,x)$ models the deflection of the fluid surface from the rest position and the equation was proposed by Whitham as an alternative to the Korteweg-de Vries (KdV) equation for the description of wave motion at the surface of a perfect fluid.% by simplified evolution equations. 

Whitham emphasized that the breaking phenomena  is one of the most intriguing long-standing problems of water wave theory, and since the KdV equation can't describe breaking, he suggested \eqref{whitham_0} with the singular kernel
\begin{equation}\label{singular}
K_0 (x) = \int_{\mathbb{R}} \bigg{(}\frac{\tanh \xi}{\xi}  \bigg{)}^{1/2} e^{i\xi x} \, d \xi,
\end{equation} as a model equation combining full linear dispersion with long wave nonlinearity, and conjectured wave breaking in \eqref{whitham_0}-\eqref{singular}.

The formal approach to prove wave breaking for Whitham type equation %(and the Camassa-Holm equation)
 originated from Seliger's 
ingenious argument \cite{Se68}, i.e., tracing the dynamics of
\begin{equation}\label{infsup}
m_1 (t) : = \inf_{x \in \mathbb{R}} [u_x (t,x)] \ \text{and} \ m_2 (t) : = \sup_{x \in \mathbb{R}} [u_x (t,x)],
\end{equation}
attained at $x=\xi_1 (t)$ and $x=\xi_2 (t)$, respectively, provided that $K$ be bounded and integrable, among other hypotheses.
The mapping $t \rightarrow \xi_i (t)$, however, may be multi-valued so the curves in general are not necessarily well-defined. In addition to this, to carry out Seliger's formal analysis, one needs to assume that the
curves $\xi_1 (t)$ and $\xi_2 (t)$ are smooth. These additional strong assumptions were shown
unnecessary later by the rigorous analytical proof of Constantin and Escher \cite{CE98}.

Following the argument in \cite{Se68,CE98}, in this paper $K(x)$ is assumed to be regular (smooth and integrable over $\mathbb{R}$), symmetric and monotonically decreasing on $x \in [0, \infty)$. For non-integrable $K(x)$ case, we refer to \cite{VH17} and references therein.
Differentiating the first equation in \eqref{whitham_0} with respect to $x$ and evaluating the resulting equations at $x=\xi_1 (t)$ and $x=\xi_2 (t)$, two coupled differential \emph{inequalities} are deduced in \cite{Se68}:
\begin{subequations} \label{m_eqn_k}
    \begin{equation} \label{m1_eqn_k}
       \frac{dm_1}{dt} \leq -m^2 _1 (t) + K(0)(m_2 (t) -m_1 (t)) \ \ a.e.,
    \end{equation}
    \begin{equation} \label{m2_eqn_k}
        \frac{dm_1}{dt} \leq -m^2 _2 (t) + K(0)(m_2 (t)-m_1 (t)) \ \ a.e,
    \end{equation}
\end{subequations}
where $K(0)>0$.

The wave breaking condition of the Whitham type equation in \cite{Se68} is 
\begin{equation}\label{Se68con}
\inf_{x \in \mathbb{R}} [u' _0 (x)]+\sup_{x \in \mathbb{R}} [u' _0 (x)]\leq -2K(0),
\end{equation}
indeed, represents that a sufficiently asymmetric initial profile yields wave breaking in finite time. The aforementioned arguments and the rigorous analytical proof in \cite{CE98} have been considered as the cornerstone work for proving wave breaking for nonlinear nonlocal shallow water equations. Further, the condition \eqref{Se68con} has been used widely in many studies, including very recent works in \cite{GH18, YL20}. This is because it preserves a useful structure for the proof: if $m_1 (0) +m_2 (0) \leq -2K(0) $, then this relation remains so for all time.

The main contribution of this study is extending \eqref{Se68con} into a larger set, thereby obtaining a lower threshold for the wave breaking and extending the works in several aforementioned papers. Also, we provide an upper bound of wave breaking time. Our proof is base on simple phase plane analysis equipped with delicate time estimates, e.g \cite{YL22}.

To state our main result, for the sake of simplicity, we let $K(0):=1$ then \eqref{m_eqn_k} is reduced to
\begin{subequations} \label{m_eqn}
    \begin{equation} \label{m1_eqn}
       m' _1 (t) \leq -m^2 _1 (t) + m_2 (t) -m_1 (t) \ \ a.e.,
    \end{equation}
    and
    \begin{equation} \label{m2_eqn}
        m' _2 (t) \leq -m^2 _2 (t) + m_2 (t) -m_1 (t) \ \ a.e.
    \end{equation}
\end{subequations}
From \eqref{infsup}, we necessarily have $m_1 (t) \leq 0 \leq m_2 (t)$, as long as they exist. 

We now let $$\Omega:=\{(m_1 , m_2) \ | \ m_1 <-2 \ \text{and} \ 0\leq m_2 < m^2 _1 + m_1  \},$$
as shown in Figure \ref{fig1}, and state the main theorem.

%For notational convenience, in our proof, we sometimes omit the notation for dependence on $t$.

\begin{figure}[ht]
\begin{center}
\includegraphics[width=120mm]{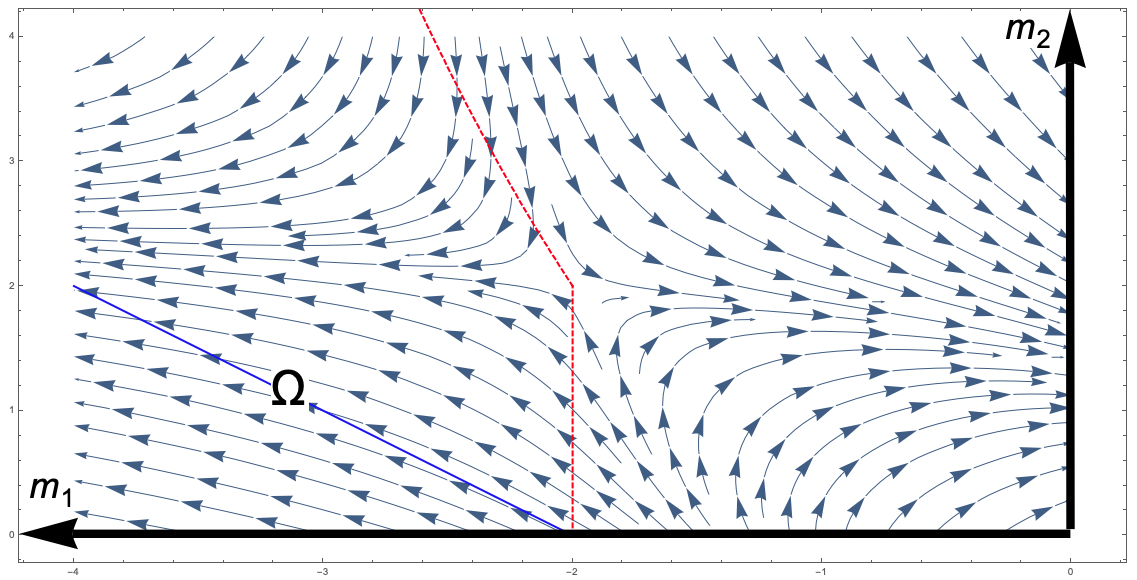}
\end{center}
\caption{The area in the left of the red lines represents $\Omega$ in Theorem \ref{main_thm}. The region below the blue line represents the wave breaking condition in Seliger's work \cite{Se68}. Small arrows display phase diagrams of the corresponding system of differential \emph{equations} in \eqref{corrdpq}.  }\label{fig1}
\end{figure}

\begin{theorem}\label{main_thm}
Consider \eqref{whitham_0}. If $u_0 \in H^{\infty}$ satisfies
$$(\inf_{x \in \mathbb{R}}[u' _0(x)] , \sup_{x \in \mathbb{R}}[u' _0 (x)]) \in \Omega,$$
then the solution must develop wave breaking before $T^*$, with
$$T^* =\frac{1}{2}\log\bigg{(} \frac{m_1 (0)}{2+m_1 (0)} \bigg{)} + \max\bigg{\{}0, \frac{m_1 (0) + m_2 (0)}{2m_1 (0)(2+ m_1 (0))}\bigg{\}}.$$
%then the solution must develop wave breaking in a finite time, and we we observe wave breaking before some
%$$T<\frac{1}{2}\log\bigg{(} \frac{m_1 (0)}{2+m_1 (0)} \bigg{)} + \max\bigg{\{}0, \frac{m_1 (0) + m_2 (0)}{m_1 (0)(2+ m_1 (0))}\bigg{\}}$$
%we observe wave breaking before
\end{theorem}
\break
Several remarks are in order regarding this result.

1. It is easy to see  that $T^* > 0$ on $\Omega$, because $m_1 (0)<-2$.

2. The wave breaking condition in Seliger \cite{Se68} is
$$m_1 (0) + m_2 (0) \leq -2,$$
which is represented in Figure \ref{fig1}. One can see that the condition in Theorem \ref{main_thm} is somewhat \emph{optimal}. %Indeed, consider a threshold  $m_2 = g(m_1)$ for wave breaking (if any). On this threshold, \eqref{m1_eqn} is reduced to 
%$$m' _1 (t) \leq -m^2 _1 (t) + g(m_1 (t)) -m_1 (t),$$
%so that $g(m_1)$ is at most $m_1 ^2 + m_1$ to allow $m_1 ' \leq 0$,\\
Indeed, suppose that there is a threshold  $m_2 = g(m_1)$ for wave breaking where $g(m_1) > m^2 _1 + m_1$. We want to $(m_1 (t), m_2 (t))$ stays on one side of the threshold, that is $m_2 (0) \leq g(m_1 (0))$ implies $m_2 (t) \leq g(m_1 (t))$ for all $t$. However, on this threshold, \eqref{m1_eqn} is reduced to 
$$m' _1 (t) \leq -m^2 _1 (t) + g(m_1 (t)) -m_1 (t).$$
We see that the right hand side of the above inequality is positive. Thus, $m' _1 (t)$ can be positive on the threshold, so that $(m_1 (t), m_2 (t))$  may cross the threshold from one side to other side.
%so that $g(m_1)$ is at most $m_1 ^2 + m_1$ to allow $m_1 ' \leq 0$,\\

3. Sometimes, there are some monotonicity relations between a system of differential \emph{inequalities} and a corresponding system of differential \emph{equations} e.g., \cite{BL20, CT09}. However, we note that there is no direct comparison between \eqref{m_eqn} and the corresponding system of differential \emph{equations}
\begin{equation}\label{corrdpq}
x'=-x^2 + y-x, \ \ y' = -y^2 + y-x.
\end{equation}
 This means that even if we find a threshold condition for \eqref{corrdpq}, that condition may not work for \eqref{m_eqn}.
 
4. There are some generalizations of the system in \eqref{m_eqn},
$$m' _i  \leq -\alpha_i m^2 _i + \beta(m_2 - m_1) + \gamma, \ \ i=1,2.$$
In \cite{YL20}, the author discussed wave breaking conditions for  the above system with $\gamma\equiv 0$, and $\alpha_i \equiv \alpha_i (t)$ are time dependent functions that are allowed to change their values and signs over time.  For $\alpha_1=\alpha_2 =positive \ constant$, and $\gamma$ non-negative constant % and $c\equiv f(t)$
 case, one can consult  \cite{MLQ16}, where the authors also extended Seliger's  work, and obtained wave breaking time estimates. We should point out that under the same initial data, the $T^*$ in the present Theorem \ref{main_thm} is sharper than the one in   \cite{MLQ16}. Moreover, we present an alternative, yet simpler, proof based on phase plane analysis.

The details of the proof of Theorem \ref{main_thm} is carried out in the rest of the paper.

\section{Proof of Theorem \ref{main_thm}}
We first construct the invariant region for \eqref{m_eqn}.
\begin{lemma}\label{invariant}
Let 
$$\Omega : = \{(m_1 , m_2) \ | \ m_1 <-2, \ and \ 0\leq m_2 < m^2 _1 +m_1\}.$$
If $(m_1(0) , m_2(0)) \in \Omega$, then $(m_1(t) , m_2(t)) \in \Omega$ for all $t>0$.
\end{lemma}
\begin{proof}
From the phase plane of the corresponding system of differential \emph{equations}, (see Figure \ref{fig1}) one can easily see that $\Omega$ forms an invariant space for \eqref{corrdpq}. One may expect the samething holds for \eqref{m_eqn}, but we rigorously prove this, as there is no direct comparison between \eqref{corrdpq} and \eqref{m_eqn}. 

We first show that if $m_2 (0) < m^2 _1(0)+m_1 (0)$, then it remains so for all time, as long as $m_1 (t) <-2$. Let
$$f(t):=m^2 _1 (t) + m_1 (t) - m_2 (t).$$
It suffices to show that $f(0)>0$ implies $f(t)>0$ for all $t > 0$. Suppose $t_1$ is the earliest time when this
assertion is violated. Then
$f(t_1)=m^2 _1 (t_1) + m_1 (t_1) - m_2 (t_1)=0.$ Further, since $f(t)>0$ for $t<t_1$ and $f(t_1)=0$, we have
$$f' (t_1) \leq 0.$$
Consider
\begin{equation*}
\begin{split}
f'(t)&= 2m_1 (t)m' _1 (t) + m' _1 (t) - m' _2 (t)\\
&=\big{(}2m_1 (t) +1 \big{)} m'  _1 (t) - m' _2 (t)\\
&\geq (2m_1 (t) + 1) (-m^2 _1 (t) + m_2 (t)- m_1 (t)) - (-m^2 _2 (t) +m_2 (t) -m_1 (t)).
\end{split}
\end{equation*}
Once we evaluate the above inequality at $t=t_1$, we obtain
\begin{equation*}
\begin{split}
f'(t_1) &\geq (2m_1 (t) + 1) \cdot 0 - (-m^2 _2 (t_1) +m_2 (t_1) -m_1 (t_1))\\
&=(m_1 (t_1))^3 (2+ m_1 (t_1))\\
&>0,
\end{split}
\end{equation*}
where the last inequality holds because $m_1 (t_1) < -2$. This gives the contradiction.

Now we show that if $(m_1 (0), m_2 (0))\in \Omega$, then $m_1 (t)<-2$ for all time. 
%$m_2 (t) \neq 2$ for all time if $m_1(t) =-2$. This is because
Suppose not, that is, let $t_2 >0$ be the earliest time when the assertion is violated. Then $m_1 (t_2)=-2$. Also, since $m_1 (t)<-2$ for $t<t_2$ and $m_1 (t_2) =-2$, we have
$$m' _1 (t_2) \geq 0.$$
However, from \eqref{m1_eqn},
$$m' _1 (t_2) \leq -4 + m_2 (t_2) +2<0,$$
because $m_2 (t_2) <2$. 
(we can exclude $m_2 (t_2)=2$ case because, since $(m_1 (0), m_2 (0))\in \Omega$, we see that $(m_1 (t), m_2(t))\neq (-2,2)$ for all $t$, unless $(m_1 (t), m_2 (t))$ exists $\Omega$ through the line segment $\{(m_1 , m_2) \ | \ m_1 =-2, \ 0 \leq m_2 <2 \}$. Indeed, on $B_{\epsilon}(-2,2) \cap \Omega$, one can see that $m' _1 <0$. Here,  $B_{\epsilon}(-2,2)$ is the disk with radius $\epsilon$ centered at $(-2,2)$.) This gives the contradiction and completes the proof.
\end{proof}

\begin{figure}[ht]
\begin{center}
\includegraphics[width=120mm]{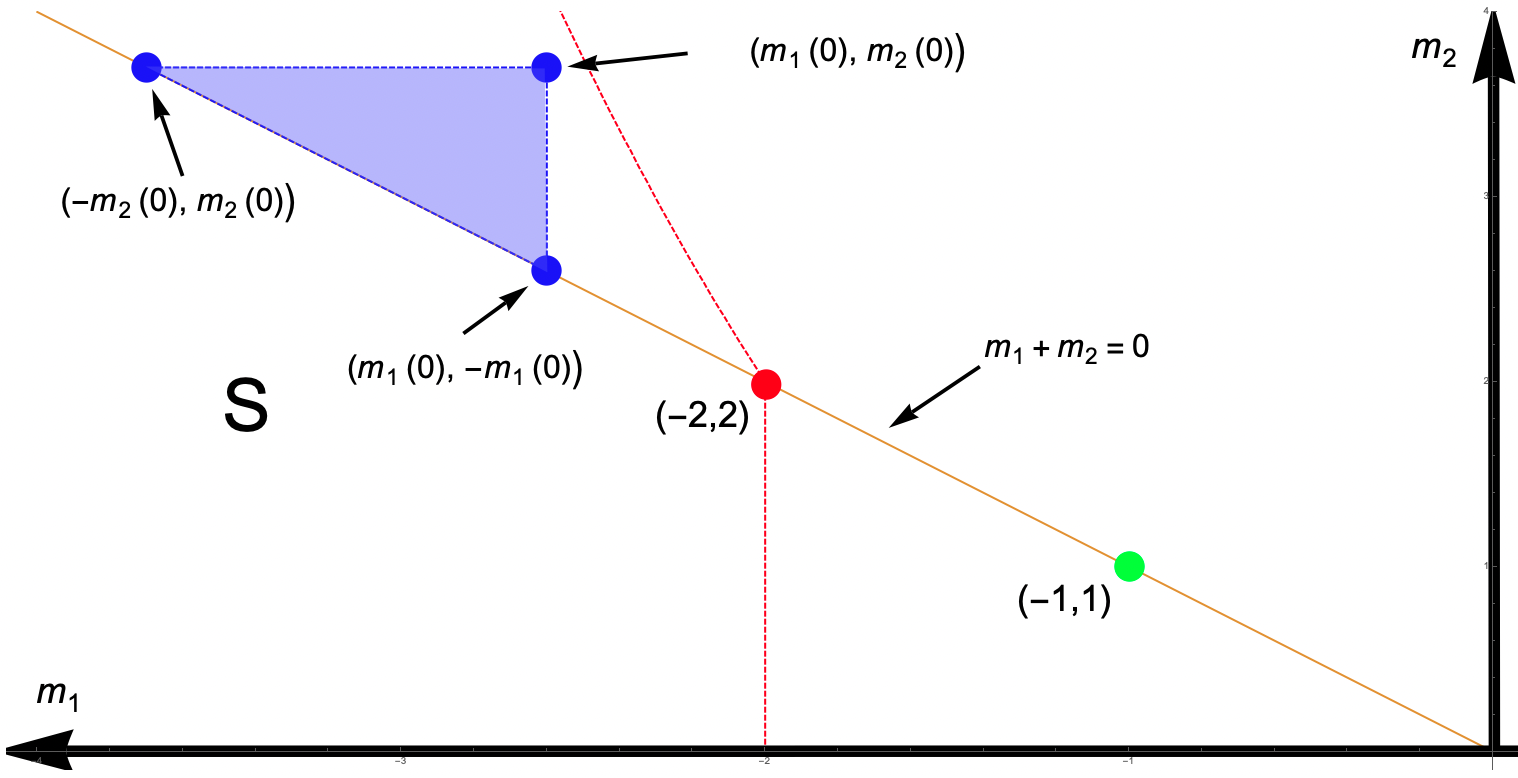}
\end{center}
\caption{Region $S$ and the triangle used in Lemma \ref{time}}\label{fig2}
\end{figure}

%\begin{proof}
%Consider some function $m_2 = g(m_1)$ on $m_1 <-2$. On $m_2 = g(m_1)$, () and () lead to
%$$m' _1 \leq -m^2 _1 + g(m_1) - m_1,$$
%and
%$$m' _2 \leq - (g(m_1))^2 + g(m_1) - m_1,$$
%respectively.
%We require $g(m_1) = m^2 _1 + m_1$ so that 
%$$m' _1 \leq 0,$$
%and
%$$m'_2 \leq - (m^2 _1 + m_1)^2 + (m^2 _1 + m_1) -m_1 = -m^3 _1 (2+m_1) \leq 0,$$
%when $m_1 \leq -2$.
%\end{proof}

We show that if $(m_1 (0), m_2(0)) \in \Omega$, then $m_1 (t) + m_2 (t) \leq 0$ in finite time, and it remains so for all time.
\begin{lemma}\label{time}
If $(m_1 (0), m_2(0)) \in \Omega$, then there exists $t^* \in [0, \infty)$ such that
$$m_1 (t) + m_2 (t) \leq 0,$$
for all $t\geq t^*$, where
\begin{equation}\label{tstar}
t^* = \max\bigg{\{}0, \frac{m_1 (0) + m_2 (0)}{2m_1 (0)(2+ m_1 (0))}\bigg{\}}.
\end{equation}
\end{lemma}
\begin{proof}
Suppose $(m_1 (0), m_2(0)) \in \Omega$. 

It is easy to see that if 
$$m_1 (0)+m_2 (0)\leq 0$$
already, then it remains so for all time. Indeed, for a fixed $t$, the inequalities in \eqref{m_eqn} with $m_1 (t)+m_2 (t)=0$ give
$$m' _i (t)\leq  -m^2 _1 (t) -2m_1(t) = - m_1 (t) (m_1 (t) +2) <0,$$
for $i=1,2$. Here the last inequality holds because $m_1 (t) < -2$ due to Lemma \ref{invariant}. Summing up, we get
$$(m_1 (t) + m_2 (t))' <0.$$ This proves the claim. In addition to this, defining $$S : = \Omega \cap \{(m_1 , m_2) \ | \  m_1 + m_2 \leq 0\},$$ we see that together with Lemma \ref{invariant} the above assertion states that $S$ forms an invariant space for \eqref{m_eqn}. See Figure \ref{fig2}.

Now we consider the case $m_1 (0) + m_2 (0) >0$  (i.e., $(m_1 (0), m_2 (0))\in \Omega \cap S^{c}$). For a fixed $t$, any $(m_1(t) , m_2(t)) \in \Omega \cap S^{c}$ leads to
$$m' _1 \leq -m^2 _1 + m_2 - m_1 < - m^2 _1 + (m^2 _1 + m_1) -m_1=0,$$
and
$$m' _2 \leq -m^2 _2 + m _2 -m_1 < -m^2 _2 + m_2 + m_2 = -m_2 (m_2 -2) <0.$$
Thus, both $m_1$ and $m_2$ are strictly decreasing on $\Omega \cap S^{c}$. Applying this, we find that any $(m_1 (t) , m_2 (t))$ initiated from $(m_1 (0), m_2 (0))\in \Omega \cap S^{c}$ stays in the triangular region 
$$\lr^{(m_1 (0), m_2 (0))}$$
with vertices $(m_1 (0), m_2 (0))$, $(-m_2 (0), m_2 (0))$ and $(m_1 (0), -m_1 (0))$, unless $(m_1 (t) , m_2 (t))$ touches $m_2 + m_1 =0$ line (see Figure \ref{fig2}). In other words, $(m_1 (t) , m_2 (t))$ may exit the triangular region only through the hypotenuse of the triangle.

We shall show that $(m_1 (t) , m_2 (t))$ touches $m_2 + m_1 =0$ line in finite time. On the triangle  $\lr^{(m_1 (0), m_2 (0))}$, \eqref{m_eqn} leads to
\begin{equation*}
\begin{split}
(m_1 + m_2)' &\leq -m^2 _1 -2m_1 - m^2 _2 +2m_2\\
&=-(m_1 +1)^2 - (m_2 -1)^2 +2\\
&\leq \max_{(m_1 ,m_2)\in \lr } \big{\{}-(m_1 +1)^2 - (m_2 -1)^2 +2  \big{\}}\\
&= -\min_{(m_1 ,m_2)\in \lr } \big{\{}(m_1 +1)^2 + (m_2 -1)^2   \big{\}}+2\\
&=-(m_1 (0)+1)^2 - (-m_1 (0) -1)^2 +2,
\end{split}
\end{equation*}
where the last equality holds because the closest point to $(-1,1)$ (the green dot in Figure \ref{fig2}) on the  the triangle $\lr^{(m_1 (0), m_2 (0))}$ is $(m_1 (0), -m_1 (0))$. Thus, since $m_1 (0)<-2$, we see that
$$(m_1+ m_2)' \leq -2m_1 (0)(2+ m_1 (0))<0,$$
on $\lr^{(m_1 (0), m_2 (0))}$. Integration yields,
$$m_1 (t) + m_2 (t) \leq -2m_1 (0)(2+ m_1 (0))t + m_1 (0) + m_2 (0).$$
Hence we find that $(m_1 (0), m_2 (0))\in \Omega \cap S^{c}$ leads to $m_2 (t_*) + m_1 (t_*) =0$ for some 
$$t_* \leq \frac{m_1 (0) + m_2 (0)}{2m_1 (0)(2+ m_1 (0))},$$
which is positive. This gives the desired result.
\end{proof}
%where the last equality holds because the closest point to the circle $(m_1 +1)^2 + (m_2 -1)^2 -2=0$ on the triangle $\lr^{(m_1 (0), m_2 (0))}$ is $(m_1 (0), -m_1 (0))$.

%now we show it thouches the line in finite time!

%\begin{align}
% A &= \ll\\
% B &= \lr
%\end{align}

$$$$

Now we are ready for the last step of proving Theorem \ref{main_thm}. Let $(m_1 (0), m_2 (0))\in \Omega$, then by Lemma \ref{time},
$$m_2 (t)+m_1 (t)\leq 0, \ \text{for all} \ t\geq t^*,$$
where $t^*$ is given in \eqref{tstar}. We note that if $m_2 (0)+m_1 (0)\leq 0$, then $t^* =0$.  Thus, when $t \geq t^*$, from \eqref{m1_eqn},
$$m' _1 \leq -m^2 _1 + m_2 -m_1\leq -m^2 _1 -m_1 -m_1 = -m_1 (m_1 +2),$$
and integration yields
$$\frac{1}{m_1 (t)}\geq \frac{1}{2} \bigg{\{} e^{2(t-t^*)}\bigg{(}  \frac{2}{m_1 (t^*)} +1 \bigg{)}\bigg{\}}-\frac{1}{2},$$
so that
$$m_1 (t) \rightarrow -\infty,$$
before $t$ reaches 
$$\frac{1}{2}\log\bigg{(} \frac{m_1 (t^*)}{2+m_1 (t^*)} \bigg{)} +t^*.$$
Finally, since $m_1 (t^*) \leq m_1 (0) <-2$, we obtain
$$\frac{1}{2}\log\bigg{(} \frac{m_1 (t^*)}{2+m_1 (t^*)} \bigg{)} +t^* \leq \frac{1}{2}\log\bigg{(} \frac{m_1 (0)}{2+m_1 (0)} \bigg{)} +t^* =: T^*, $$
and this proves Theorem \ref{main_thm}.

%\section{Acknowledgements and funding statement }
%The author has no competing interests.There is no external funding associated with this paper.

$$$$

\bibliographystyle{abbrv}

\end{document}